\documentclass{amsart}

\usepackage{amsmath}
\usepackage{amssymb}
\usepackage{amscd}
\usepackage{epsfig}
\usepackage{diagrams} 

\newtheorem{prop}{Proposition} 
\newtheorem{lem}{Lemma} 

\newtheorem{thm}{Theorem}
\newtheorem{defn}{Definition}

\DeclareMathOperator{\OO}{\mathcal{O}}
\DeclareMathOperator{\PP}{\mathbb{P}}

\DeclareMathOperator{\LL}{\mathcal{L}}
\DeclareMathOperator{\C}{\mathbb{C}}
\DeclareMathOperator{\Hom}{\mathrm{Hom}}

\DeclareMathOperator{\E}{\mathbb{E}}

\begin{document}

\title{A Combinatorial Case of the Abelian-Nonabelian Correspondence}
\author{K. Taipale}

\begin{abstract} The abelian-nonabelian correspondence outlined in \cite{BCK2} gives a broad conjectural relationship between (twisted) Gromov-Witten invariants of related GIT quotients. This paper proves a case of the correspondence explicitly relating genus zero $m$-pointed Gromov-Witten invariants of Grassmannians $Gr(2,n)$ and products of projective space $\PP^{n-1} \times \PP^{n-1}$. Computation of the twisted Gromov-Witten invariants of $\PP^{n-1} \times \PP^{n-1}$ via localization is used.
\end{abstract}

\maketitle

\tableofcontents

\section{Introduction}\label{sec: intro} This paper presents a proof of a particular case of the abelian-nonabelian correspondence for genus zero cohomological Gromov-Witten theory. In general, the abelian-nonabelian correspondence states a conjectural relationship between the Gromov-Witten theories of $V//G$ and $V//T$ -- here $V$ is a smooth projective variety with a (linearized) group action by a Lie group $G$ and $T \subset G$ is its maximal torus. In \cite{BCK1}, localization techniques were used to prove a relationship between the J-functions of the Grassmannian $G(k,n)$ and product of projective spaces $(\PP^{n-1})^k$, and a proof in quantum cohomology gave an explicit formula relating their three-point genus zero Gromov-Witten invariants. Neither result, though, implies the correspondence for $m$-pointed Gromov-Witten invariants. Papers \cite{BCK2} and \cite{CKS} expand the ideas of \cite{BCK1} to more general GIT quotients $V//G$ and $V//T$ as above. Conjecture (4.2) in \cite{BCK2} suggests 

\begin{displaymath}
\langle \gamma_1, \ldots, \gamma_m \rangle_{0,m,d}^{V//G} = \frac{1}{|W|} \sum_{\tilde{d} \mapsto d} \langle \tilde{\gamma_1}, \ldots, \tilde{\gamma_m} \rangle_{0,m,\tilde{d}}^{V//T, E}
\end{displaymath}  where $\tilde{\gamma_i}$ is a lift of $\gamma_i$ from $V//G$ to $V//T$ and $W$ is the Weyl group of $T \subset G$. Inspired by \cite{BCK1, BCK2}, this paper uses localization techniques to prove this conjecture for $m$-point Gromov-Witten invariants of $G(2,n)$ and $(\PP^{n-1})^2$. Here $V = \Hom ( \mathbb{C}^2, \mathbb{C}^n)$, $2 < n$, and $G = GL_2(\mathbb{C})$. $G$ acts on $V$ by multiplication, and the quotients $V//G = Gr(2,n)$ and $V//T = (\mathbb{P}^{n-1})^2$ are the nonabelian and abelian sides of the correspondence. 

The main result in this paper is the following:
\begin{thm}\label{thm: main} The genus zero cohomological Gromov-Witten invariants of $Gr(2,n)$ and $(\PP^{n-1})^2$ are related by the formula
\begin{displaymath}
\langle \gamma_1, \ldots, \gamma_m \rangle_{0,m,d}^{Gr(2,n)} = \frac{1}{2} \sum_{d_1+d_2= d} \langle \tilde{\gamma}_1, \ldots, \tilde{\gamma}_m \rangle_{0,m,(d_1,d_2)}^{(\mathbb{P}^{n-1})^2, E}\end{displaymath} with $E = \oplus_{\alpha \in R} L_{\alpha}$ and $R$ the root system associated to $G$ and $T$.\end{thm}  For the proof we use a modification of Kontsevich's method of summing over graphs to compute the twisted Gromov-Witten invariants. Although poles occur in the localization computation, we exploit a Weyl group action on the torus-invariant stable maps and show that the twisting by $E$ cancels these poles (if we use the appropriate torus action!).

These results reinforce the idea that in nice cases combinatorial tools are a useful way to ``get one's hands dirty'' and demonstrate explicit geometric relationships between Gromov-Witten invariants. One might wonder how this would extend to toric varieties and flag varieties, which are related by the abelian non-abelian correspondence and also have robust combinatorial structures. However, the difficulty of extending these techniques even to $G(k,n)$ and $(\PP^{n-1})^k$ for $k >2$ points out the drawbacks of relying too heavily on the particulars of such structures.

\section{Setup}\label{sec: setup} 
\subsection{Spaces and lifting}\label{sec: spaces} The setting is as in \cite{BCK1, BCK2}. For homogeneous varieties there is a rational map $\varphi$ from $V//T$ to $V//G$. The rational map from $(\mathbb{P}^{n-1})^k$ to $Gr(k,n)$ is given by \begin{displaymath}\varphi: [a_1, \ldots, a_k] \mapsto \mathrm{span}(a_1, \ldots, a_k)\end{displaymath} for $a_i$ a point in the $i$th factor of $(\mathbb{P}^{n-1})^k$, as long as we exclude $k$-tuples $[a_1, \ldots, a_k]$ for which some $a_i$ are linearly dependent. Let $D$ denote the unstable locus of dependent $a_i$s and let $U$ denote the good locus of independent $a_i$s. Notice that the map $\varphi$ is invariant under the action of the Weyl group of $GL_k(\mathbb{C})$. In this situation the Weyl group is simply $S_k$, the symmetric group on $k$ elements, and acts by permuting the factors in $(\PP^{n-1})^k$. \cite{Martin} looked at cohomology of symplectic spaces related in this way, and \cite{BCK1,BCK2} considered the situation in algebraic geometry and gave it the name of ``the abelian-nonabelian correspondence.''

The relationship between $V//G$, $V//T$, and the open set $U$ can be illustrated as follows, where $i$ is an open immersion and $\varphi$ is a fiber bundle with fiber $G/T$:

\begin{diagram} 
U &\rInto_i &V//T  \\
\dTo_{\varphi} & & \\
V//G &&
\end{diagram} For $V//G = Gr(2,n)$ and $V//T = (\mathbb{P}^{n-1})^2$, $\varphi$ has affine fibers. 

The map $\varphi$ allows us to lift cohomology classes in $V//G$ to cohomology classes in $V//T$. For $V//G = Gr(2,n)$ and $V//T = (\mathbb{P}^{n-1})^2$, we use the Schubert basis $\{\sigma_{\mu}\}$ for $H^*(Gr(2,n))$. The index $\mu$ is a partition that fits (as a Young diagram) into a $2 \times (n-2)$ rectangle. The canonical lifting of $\sigma_{\mu} \in H^*(Gr(2,n))$ is given by 
\begin{displaymath} \tilde{\sigma}_{\mu} = S_{\mu}(H_1, H_2) \in H^*((\PP^{n-1})^2).\end{displaymath} The right-hand side is the Schur polynomial $S_{\mu}$ of the partition $\mu$ in terms of the hyperplane classes $H_i$ of the $i$th factor of $(\mathbb{P}^{n-1})^2$. Notice that the class $\sigma_{\mu} \in H^*(Gr(2,n))$ and the lifted class $\tilde{\sigma}_{\mu} \in H^*((\mathbb{P}^{n-1})^2)$ have the same codimension. Since Schur polynomials are symmetric and homogeneous, the lifts are $S_2$-invariant.

Martin \cite{Martin} proved the ``degree zero version'' of the abelian-nonabelian correspondence:

\begin{thm}[Martin, Theorem B] For $\gamma \in H^*(Gr(k,n))$ with lift $\tilde{\gamma} \in H^*((\mathbb{P}^{n-1})^k)$, \begin{displaymath} \int_{Gr(k,n)} \gamma = \frac{1}{k!} \int_{(\mathbb{P}^{n-1})^k} \tilde{\gamma} \cdot \prod_{i<j} (H_i-H_j)(H_j-H_i).\end{displaymath} \end{thm}

The term in the right-hand integrand modifying $\tilde{\gamma}$ gives us a hint as to what ``twisting'' should be in the quantum version.

\subsection{Twisted Gromov-Witten invariants}\label{sec: twisted}
For background on Gromov-Witten invariants and the moduli space of stable maps $\overline{M}_{0,}(X,d)$, consult the survey article \cite{FP}. Kontsevich outlined localization computation of Gromov-Witten invariants in his paper \cite{K}, and these methods will be used throughout the rest of this note. 

Usual Gromov-Witten invariants ``count'' the number of curves of degree $d$ intersecting cohomology classes $\gamma_1, \ldots, \gamma_m$ in a space $X$ by looking at stable maps $f:C \rightarrow X$ with marked points $p_1, \ldots, p_m$ whose image $f_*[C]$ is of degree $d$, and such that the image of the marked point $p_i$, $f(p_i)$, lands on $\gamma_i$. Then the usual Gromov-Witten invariant is written \begin{displaymath} \langle \gamma_1, \ldots, \gamma_m \rangle_{0,m,d}^X = \int_{[\overline{M}_{0,m} (X,d)]} ev_1^* \gamma_1 \cup \ldots \cup ev^*_m \gamma_m \end{displaymath}

Twisted Gromov-Witten invariants include in the integrand the Euler class of some obstruction bundle. As introduced by \cite{CG} and further explained in \cite{CoatesEtAl}, the geometric meaning of a twisted Gromov-Witten invariant depends on the type of obstruction bundle: both local Gromov-Witten invariants (not discussed here) and Gromov-Witten invariants of a subvariety of a variety $X$ can be obtained in this framework. Here, twisted Gromov-Witten allow us to encode the geometric relationship between Grassmannians and products of projective space in a natural way. The rational map $\varphi$ takes the open set $U \subset (\mathbb{P}^{n-1})^k$ as a $k!$-fold cover to $Gr(k,n)$, with $T \varphi|_U = E$ the obstruction bundle. 

Given this vector bundle $E$, we need the corresponding element $E_{0,m,d}$ in the Grothendieck group of vector bundles over $\overline{M}_{0,m}(X, d)$. Pull back along the evaluation map $ev$ that takes the point $(C, f, p_1, \ldots, p_{m+1}) \in \overline{M}_{0,m+1}(X,d)$ to $f(p_{m+1}) \in X$ and push forward along the map forgetting the $(m+1)$th point, $\pi$:

\[\begin{CD}
\overline{M}_{0,m+1}(X, d) @>ev>> X \\
@VV{\pi}V    \\
\overline{M}_{0,m}(X, d)
\end{CD}\]
Using the K-theoretic pushforward, this results in the complex  $E_{0,m,d} := R^{\bullet}\pi_* ev^* E$. We compute the cohomological Euler class of this complex in section \ref{sec: eqtwist}. Then we define the \textit{twisted Gromov-Witten invariant} by: \begin{displaymath} \langle \gamma_1, \ldots, \gamma_m \rangle_{0,m,d}^{X,E} = \int_{[\overline{M}_{0,m} (X,d)]} ev_1^* \gamma_1 \cup \ldots \cup ev^*_m \gamma_m \cup e(E_{0,m,d}) \end{displaymath}

\section{Using localization}

The torus action on any scheme $X$ induces a torus action on $\overline{M}_{0,m}(X, d)$. Bott's localization formula can be used to shift from integration over $[\overline{M}_{0,m}(X, d)]$ to integration over the torus-fixed locus in $\overline{M}_{0,m}(X, d)$. When $X$ has isolated torus-fixed points and isolated torus-invariant one-dimensional orbits, the components of the torus-fixed locus of $\overline{M}_{0,m}(X, d)$ can be indexed by graphs. Moreover, the contributions of each component can be calculated directly from the graph and the insertions of the Gromov-Witten invariant. Thus we can transform a geometric problem into a combinatorial problem, calculating the contribution to the Gromov-Witten invariant of each component and summing over the possible graphs.

The $(\C^*)^n$ action on $Gr(2,n)$ gives a $(\C^*)^n$ action on $\overline{M}_{0,m}(Gr(2,n), d)$ as well as corresponding to a $(\C^*)^n$-action on $(\PP^{n-1})^2$. However, this $(\C^*)^n$ action is not big enough to give isolated torus-invariant one-dimensional orbits on $\overline{M}_{0,m}((\PP^{n-1})^2, (d_1, d_2))$. Instead, in what follows we consider a $((\C^*)^n)^2$-action and specialize to the diagonal $(\C^*)^n$-action for our final results.

By localization, for each $(d_1, d_2)$ we can rewrite each twisted Gromov-Witten invariant \[\frac{1}{2} \langle \tilde{\gamma_1}, \ldots, \tilde{\gamma_m} \rangle_{0,m,(d_1,d_2)}^{(\mathbb{P}^{n-1})^2, E}\]  on the right-hand side of Theorem (\ref{thm: main}) as a sum over $((\C^*)^n)^2$-fixed loci in the moduli space $\overline{M}_{0,m}((\PP^{n-1})^2, (d_1, d_2))$. These $((\C^*)^n)^2$-fixed loci are indexed by graphs $\Gamma$. We write \[ \frac{1}{2} \sum_{d_1+d_2= d} \langle \tilde{\gamma_1}, \ldots, \tilde{\gamma_m} \rangle_{0,m,(d_1,d_2)}^{(\mathbb{P}^{n-1})^2, E} = \frac{1}{2}\sum_{\Gamma \subset U} \frac{I(\Gamma) T(\Gamma)}{e(N_{\Gamma})} + \frac{1}{2}\sum_{\Gamma \cap D \neq \emptyset} \frac{I(\Gamma) T(\Gamma)}{e(N_{\Gamma})}.\] Here $I(\Gamma)$ indicates evaluation of insertions $\tilde{\gamma}_i$, $T(\Gamma)$ the Euler class of the twisting bundle $E_{0,m,d}$, and $e(N_{\Gamma})$ the Euler class of the bundle $N_{\Gamma}$ normal to $\overline{M}_{\Gamma}$. Notice that the sum over graphs has been split into two parts. We abuse notation to represent those graphs corresponding to maps whose image lies in $U$ by $\Gamma \subset U$, even though $\Gamma$ is in $\overline{M}_{0,m}((\PP^{n-1})^2, (d_1, d_2))$ and $U$ is in $(\PP^{n-1})^2$. Similarly, to represent those graphs corresponding to maps whose image intersects $D$ we abusively write $\Gamma \cap D \neq \emptyset$.

We claim that the Gromov-Witten invariant $\langle \gamma_1, \ldots , \gamma_m\rangle^{Gr(2,n)}_{0,m,d}$ is equal to the localization contribution from graphs representing maps to $(\PP^{n-1})^2$ that do not pass through the diagonal $D \subset (\PP^{n-1})^2$, evaluated with the appropriate insertions lifted from $Gr(2,n)$. For any graph, evaluating insertions in the Grassmannian and evaluating the corresponding lifted insertions in the cohomology of $(\PP^{n-1})^2$ gives an equal contribution if the $(\C^*)^n$-action is used (proved later on). Compare the Euler class contributions for graphs representing stable maps to the Grassmannian and corresponding ``lifted graphs'' representing stable maps to the product of projective spaces:

\begin{lem}\label{lemma: equality} As K-theory classes restricted to the components of the fixed point locus in $\overline{M}_{0,m}((\PP^{n-1})^2,(d_1,d_2))$, \begin{displaymath}R^{\bullet} \pi_*ev^* T_{(\PP^{n-1})^2} =  R^{\bullet}\pi_*ev^* \varphi^* T_{Gr(2,n)} \oplus R^{\bullet} \pi_* ev^* (E).\end{displaymath} \end{lem} 
\begin{proof} By definition, $E|_U = T\varphi$, where $T\varphi$ is the kernel of $T_{(\PP^{n-1})^2} \rightarrow \varphi^* T_{Gr(2,n)}$. The exactness of the corresponding sequence leads to the sum in K-theory.\end{proof}

The Euler classes of each side are then also equal, and so  \[\frac{e(N_{\tilde{\Gamma}}^{\overline{M}_{0,m}((\PP^{n-1})^2,(d_1,d_2))})}{T(\Gamma)} \big|_{(\C^*)^n}= e(N_{\Gamma}^{\overline{M}_{0,m}(Gr(2,n),d)})\] where we again abuse notation and let $\Gamma$ represent both a $(\C^*)^n$-fixed locus in $\overline{M}_{0,m} (Gr(2,n), d)$ and one of its lifts. Since this considers only graphs $\Gamma$ contained in the good locus $U$ of $(\PP^{n-1})^2$, for which there are exactly two distinct liftings, this gives the equality \begin{displaymath}\langle \gamma_1, \ldots, \gamma_m \rangle_{0,m,d}^{Gr(2,n)} =  \frac{1}{2}\sum_{\Gamma \subset U \subset (\PP^{n-1})^2} \frac{I(\Gamma) T(\Gamma)}{e(N_{\Gamma})}|_{(\C^*)^n} . \end{displaymath}

What remains is to prove that the second part of the sum contributes zero: \[\frac{1}{2}\sum_{\Gamma \cap D \neq \emptyset} \frac{I(\Gamma) T(\Gamma)}{e(N_{\Gamma})}|_{(\C^*)^n} = 0.\] This requires the following steps:

\begin{itemize}
\item First, for any $\Gamma$ indexing a $((\C^*)^n)^2$-fixed locus of $\overline{M}_{0,m}((\PP^{n-1})^2, (d_1,d_2))$, compute $e(N_{\Gamma})$.
\item Compute the associated contribution of equivariant twisting, $T(\Gamma)$.
\item Show that the sum of contributions from $\Gamma$ with vertices in the diagonal $D \subset (\PP^{n-1})^2$ vanishes when the action of $((\C^*)^n)^2$ is specialized to the action of the small torus $(\C^*)^n$.
\end{itemize}

\subsection{Euler class $e(N_{\Gamma})$ on the product of projective spaces}\label{sec: euler}
On $Gr(2,n)$ and $(\PP^{n-1})^2$ the torus $(\mathbb{C}^*)^n$ acts with isolated torus-fixed points but leaves families of non-isolated torus-invariant orbits in the unstable locus of $(\mathbb{P}^{n-1})^2$. Moreover, this leaves non-isolated torus-invariant loci in $\overline{M}_{0,m}((\PP^{n-1})^2, (d_1, d_2))$. To apply localization easily, instead consider the action of the big torus $((\mathbb{C}^*)^n)^2$ on $(\mathbb{P}^{n-1})^2$: the $i$th factor of $((\mathbb{C}^*)^n)^2$ acts on the $i$th factor of $(\mathbb{P}^{n-1})^2$ componentwise. This action gives isolated torus-fixed points and isolated one-dimensional torus-invariant orbits in both $(\PP^{n-1})^2$ and its moduli space of stable maps. Write elements of $((\mathbb{C}^*)^n)^2$ as \begin{equation*} \lambda = [\lambda_1^1, \ldots, \lambda_n^1]  \times [\lambda_1^2,  \ldots, \lambda_n^2] \end{equation*} with action \begin{equation*} \lambda \cdot (a_1, a_2) =  [\lambda_1^1 a_1^1, \ldots, \lambda_n^1a_n^1]  \times [\lambda_1^2a_1^2,  \ldots, \lambda_n^2a_n^2] \end{equation*} on points $(a_1, a_2) \in (\PP^{n-1})^2.$ We can then specialize the action of $((\mathbb{C}^*)^n)^2$ on $(\mathbb{P}^{n-1})^2$ to the diagonal action of $(\mathbb{C}^*)^n$ by erasing superscripts of $\lambda^i_j$!

Here we describe the graphs $\Gamma$ for $Gr(2,n)$ and $(\PP^{n-1})^2$ with the actions of $(\C^*)^n$ and $((C^*)^n)^2$, respectively. Let $(C, f, p_1, \ldots, p_m)$ be a stable map whose image $f_*[C]$ is torus-invariant for the appropriate torus $T$, and thus a torus-fixed point in $[\overline{M}_{g,m}(X, \beta)]$. Construct the graph $\Gamma$ as follows:

\begin{itemize}
\item Vertices: The $(\C^*)^n$-fixed points of $Gr(2,n)$ are the $2$-planes in $n$-space spanned by $\langle e_{i}, e_{j} \rangle$. We label $\langle e_i, e_j \rangle$ by $\langle i j \rangle$. For projective space $\PP^{n-1}$, the points $q_0 = [1:0:\cdots:0]$ through $q_{n-1} = [0:0: \cdots:0:1]$ are fixed under wither $(\C^*)^n$ or $((\C^*)^n)^2$-action. Thus torus-fixed points in $\PP^{n-1} \times \PP^{n-1}$ are points $q_i \times q_j$ which we label by $ij$. For a map to be stable and $T$-equivariant for any torus $T$, all nodes, marked points, ramification points, and contracted components of $C$ must be mapped to $T$-fixed points in the target. Label vertices of $\Gamma$ by the fixed point to which they correspond. Remember that order in the label does not matter for Grassmannians; for products of projective space, order does matter.
\item Edges: The $((\C^*)^n)^2$-invariant one-dimensional orbits in $\PP^{n-1} \times \PP^{n-1}$ have homology class $[pt] \times [line]$ or $[line] \times [pt]$. They connect two vertices of form $ij$ and $ik$ or $ij$ and $kj$, respectively. For $Gr(2,n)$, the one-dimensional orbits of $(\C^*)^n$ connect two fixed points $\langle i j \rangle$ and $\langle i k \rangle$. Each edge in a graph corresponds to a rational non-contracted component $C_e$ of $C$ mapped to a one-dimensional curve $\ell$ in the target space. We label each edge $e$ with the degree $d_e$ of the map taking $C_e$ to $\ell$. 
\item Flags: Flags are pairs $(v,e)$ of a vertex and an adjacent edge. 
\end{itemize}

In addition to edges $e$ labeled with degree $d_e$ and vertices $v$ labeled with the corresponding fixed point, we must label vertices with the genus of the corresponding contracted component, if any, and with the marked points. We will adopt the convention that the lack of a label for genus indicates a rational or trivial contracted component. 

As discussed above, Bott's formula allows us to shift the computation of the twisted Gromov-Witten invariant to the torus-fixed locus in the moduli space. We must compute $e(N_{\Gamma})$ and include in our final sum the order of the automorphism group of a graph. We leave the order of the automorphism group until the end.

The normal bundle to $\overline{M}_{\Gamma}$ is given by comparing the tangent bundles of the full moduli space of stable maps and the fixed locus. \begin{displaymath} [N_{\overline{M}_{\Gamma}}] = [T_{\overline{M}_{0,m}(X, d)}] - [T_{\overline{M}_{\Gamma}}]. \end{displaymath} Computations of $[T_{\overline{M}_{0,m}(\PP^r,d)}]$ and $[T_{\overline{M}_{\Gamma}}]$ were done by Kontsevich \cite{K} for genus zero, and for general genus in \cite{GP}. Those results are used directly in the computation of $e(N_{\Gamma})$ for our product of projective spaces. Here we deal only with genus zero curves and homogeneous spaces, so the virtual fundamental class as discussed in \cite{GP} is not necessary. Thus for the remainder of the article we drop discussion of the virtual case.

In $(\mathbb{P}^{n-1})^2$, keep in mind a few useful observations:
\begin{itemize}
\item Under the action of $((\C^*)^n)^2$, each edge in the graph represents a component $C_e$ in $C$ whose image $f_*[C_{e}]$ is a curve in one factor $\mathbb{P}^{n-1}$, with homology class $(0,d_e)$ or $(d_e,0)$. We say then that the edge has degree $d_e$.
\item It is easy to simplify the formula by observing some happy combinatorial occurences: contributions of contracted trees can be written explicitly using the string and dilaton equations.
\end{itemize}

As mentioned above, the weights of the torus action are given by variables $\lambda_i^j$, where $\lambda_i^j$ is the piece of the torus acting on the $i$th coordinate in the $j$th factor of $(\mathbb{P}^{n-1})^2$. Below, $\lambda_{i(v)}^h$ will refer to the weight at the vertex $v$ in the $h$th factor of $(\mathbb{P}^{n-1})^k$. Notation like $val(v)$ and $val(F)$ denotes the valence of a vertex $v$ or the vertex associated with flag $F$, and $n(v)$ and $n(F)$ give the number of marked points landing on the vertex $v$ or vertex associated with flag $F$.

Functoriality for products of J-functions implies that Gromov-Witten theory of products of projective spaces is equivalent to the product of Gromov-Witten theories for projective spaces \cite{B}. One can explicitly calculate localization contributions using a deformation-obstruction sequence.

\begin{prop}
 The contribution $e(N_{\Gamma})$ for the substack $\overline{M}_{\Gamma}$ of torus-invariant maps $f:C \rightarrow (\PP^{n-1})^k$ is
\begin{equation} \label{eq:euler}
\frac{1}{e(N_{\Gamma})} = \displaystyle \prod_{F}\prod_{val(F) + n(F) >2} \frac{1}{(\omega^h_F - \psi_F) }
\frac{\prod_{j \neq i(F), h} (\lambda^h_{i(F)} - \lambda^h_j) \prod_{e} \frac{(-1)^{d_e}d_e^{2d_e}}{(d_e!)^2(\lambda^h_{i(v)}-\lambda^h_{i(v')})^{2d_e}} }{\prod_{\substack{a + b = d_e, h \neq h' \\ k \neq i,j, m \neq \ell}}(\frac{a}{d_e} \lambda^h_i + \frac{b}{d_e} \lambda^h_j - \lambda^h_k)  (\lambda^{h'}_{\ell}-\lambda^{h'}_m)},
\end{equation} where $\omega_F^h = \frac{\lambda_{i(F)}^h-\lambda_{j(F)}^h}{d_e}$ and $\psi_F$ is the line bundle whose fiber over a point is the cotangent space to the component associated to $F$ at the corresponding node.
\end{prop}
\begin{proof} The result follows easily from functoriality for products of J-functions and the calculations of $e(N_{\Gamma})$ outlined in \cite{K} and \cite{GP}.\end{proof}

The string and dilaton equations can be used to simplify part of this. \begin{prop} For the product of projective space $(\PP^{n-1})^2$, the equation for $1/e(N_{\Gamma})$ for any graph $\Gamma$ can be simplified by using \begin{displaymath}
\prod_{F}\prod_{val(F) + n(F) >2} \frac{1}{(\omega^h_F - \psi_F) } = \prod_v \frac{1}{\prod_{F \; at \; v} \omega^h_F} \left(\sum_{F \; at \; v} \frac{1}{\omega_F^h} \right)^{n(F)+val(F)-3}
\end{displaymath} \end{prop}
\begin{proof}
A more general statement and proof for toric varieties can be found in \cite{Spielberg}. \end{proof} 

\subsection{Equivariant twisting contribution}\label{sec: eqtwist}

We now compute $e(E_{0,m,d})$ evaluated at the torus-fixed locus in $\overline{M}_{0,m}((\PP^{n-1})^2, (d_1, d_2))$ indexed by a given graph $\Gamma$. Recall that \[R^{\bullet}\pi_* ev^* E = E_{0,m,d}\] and \[E = \oplus_{\alpha \in R} L_{\alpha} = \oplus_{i \neq j} \OO(H_i-H_j).\] Here $R$ is the root system associated to our choice of maximal torus in $G$, and $\alpha$ are the roots. $H_1$ and $H_2$ denote the hyperplane classes of the first and second factors of $(\PP^{n-1})^2$. The computation proceeds by calculating $e ([H^0 (C, f^* E)] - [H^1 (C, f^*E)])$, as the formal sum $[H^0 (C, f^* E)] - [H^1 (C, f^*E)])$ is the fiber of $E_{0,m,d}$ over a point $(C, f, p_1, \ldots, p_m)$ of the moduli space $\overline{M}_{0,m} ((\PP^{n-1})^2, (d_1,d_2))$. The torus action on $(\PP^{n-1})^2$ lifts canonically to $E$, giving a linearization \[(H_1-H_2)|_{v = jk} = \lambda_j^1 - \lambda_k^2.\] Notice, though, that when a vertex $v = ii$ is ``in'' the diagonal $D \subset (\PP^{n-1})^2$, \[(H_1 -H_2)|_{v=ii} = \lambda_i^1 -\lambda_i^2.\] When we specialize the sum of twisting contributions to the small torus acting diagonally this ends up giving a pole in the localization contribution. As a bookkeeping device to help us count the zeroes and the poles in the sum, we introduce an auxiliary $\C^*$-action that acts by dilation with weight $t$ on fibers of the vector bundle $E$. We then calculate the total Chern class $c_t(E_{0,m,d})$: since Gromov-Witten invariants are zero unless the codimension of the integrand matches the dimension of the fundamental class $[\overline{M}_{0,m} ((\PP^{n-1})^2, (d_1, d_2))]$, the top Chern class (Euler class) of $E_{0,m,d}$ is picked out by the integral.

The normalization sequence for the curve $C$ is \begin{equation} 0 \rightarrow \mathcal{O}_C \rightarrow \bigoplus_{v} \mathcal{O}_{C_v} \oplus \bigoplus_{e} \mathcal{O}_{C_e} \rightarrow \oplus_F \mathcal{O}_{C_F} \rightarrow 0 \end{equation} where $e$ are edges, $v$ vertices, and $F$ flags. Tensor the sequence with $f^*E$.

Taking $\C^* \times T$-equivariant Euler classes, where $T = ((\C^*)^n)^2$, we get \begin{displaymath} \frac{e(H^0(C,f^*E))}{e(H^1(C,f^*E))} = \frac{e(\bigoplus_{v} H^0(C_v, f^*E|_{C_v}) \oplus \bigoplus_{e} H^0(C_e,f^*E|_{C_e}))e(\oplus_F H^1(C_F,f^*E|_{C_F}))}{e(\oplus_F H^0(C_F,f^*E|_{C_F}))e(\bigoplus_{v} H^1(C_v, f^*E|_{C_v}) \oplus \bigoplus_{e} H^1(C_e,f^*E|_{C_e}))} \end{displaymath} Simplify by noting that $H^1(C_F, f^*E|_{C_F})$ is trivial for dimension reasons, and note that \begin{align}H^1(C_v, f^*E) &\cong H^1 (C_v, \OO_{C_v}) \otimes f^* E|_{v} \\ &= \E^{\vee} \otimes f^* E|_v, \end{align} where $\E$ is the Hodge bundle. Since we deal here only with genus zero curves, the Euler class of this term contributes trivially. Then \begin{displaymath} \frac{e(H^0(C,f^*E))}{e(H^1(C,f^*E))} = \frac{e(\bigoplus_{v} H^0(C_v, f^*E|_{C_v}) \oplus \bigoplus_{e} H^0(C_e,f^*E|_{C_e}))}{e(\bigoplus_F H^0(C_F,f^*E|_{C_F}))e(\bigoplus_{e} H^1(C_e,f^*E|_{C_e}))} \end{displaymath}

Use the fact that $c_1(L_{\alpha}) = -c_1(L_{-\alpha})$, and let $\Delta_{\alpha} :=c_1(L_{\alpha})$. Since $H^0(C_v , f^*E) \cong E|_v$, write the first term in the numerator as \begin{equation}e(\oplus_{v}H^0(C_v, f^*E)) =  \prod_v \prod_{\alpha \in R^{+}}(t+ \Delta_{\alpha})|_v (t-\Delta_{\alpha})|_v.\end{equation} Moving to the denominator, convert the following product over flags to one over vertices of $\Gamma$: \begin{align} e(\oplus_{F}H^0(C_F, f^*E)) &= \prod_F \prod_{\alpha \in R^+} (t+\Delta_{\alpha})|_F(t-\Delta_{\alpha})|_F \\ &= \prod_v \prod_{\alpha \in R^+} (t + \Delta_{\alpha})|_v^{val(v)}(t - \Delta_{\alpha})|_v^{val(v)} .\end{align} Last, use the following lemma to calculate the contribution from the edges:

\begin{lem}Consider $\PP^1$ with a $\C^*$-action keeping $0$ and $\infty$ fixed. Given any vector bundle $E = \LL \oplus \LL^{\vee}$ on $\PP^1$ built from $\LL$ a line bundle, \begin{equation} e \left( \frac{H^0(\PP^1, E)}{H^1(\PP^1, E)} \right) =(-1)^{\deg \LL}c_1(\LL)|_{0} c_1(\LL)|_{\infty}. \end{equation} \end{lem}

\begin{proof} Assume that $\LL \cong \OO (d)$, without loss of generality. Then $H^0(\PP^1, \LL \oplus \LL^{\vee}) \cong H^0(\PP^1, \LL) $ is a vector space of dimension $d+1$. The linearization of $\LL$ gives weights $c_1(\LL)|_0$ and $c_1(\LL)|_{\infty}$, and so \begin{equation}e(H^0(\PP^1, \LL) ) = \prod_{0 \leq i \leq d} \frac{(d-i)c_1(\LL)|_0 + i c_1(\LL)|_{\infty}}{d}.\end{equation} Use Serre duality to see that \begin{align} H^1(\PP^1, \LL \oplus \LL^{\vee}) &\cong H^1(\PP^1, \LL^{\vee}) \\ & \cong H^0(\PP^1, \LL \otimes \OO(-2))^{\vee}.\end{align} Thus \begin{equation}e(H^1 (\PP^1, \LL \oplus \LL^{\vee})) = (-1)^d \prod_{0 <i < d} \frac{(d-i)c_1(\LL)|_0 + i c_1(\LL)|_{\infty}}{d}.\end{equation} Take the ratio desired: \begin{align}e \left( \frac{H^0(\PP^1, E)}{H^1(\PP^1, E)} \right) &= \frac{ \prod_{0 \leq i \leq d} \frac{(d-i)c_1(\LL)|_0 + i c_1(\LL)|_{\infty}}{d}}{(-1)^d\prod_{0 <i < d} \frac{(d-i)c_1(\LL)|_0 + i c_1(\LL)|_{\infty}}{d}} \\ &= (-1)^d c_1(\LL)|_{0} c_1(\LL)|_{\infty}.\end{align}\end{proof}

Use this lemma to compute the $\C^* \times T$-equivariant $e \left( \frac{H^0(C_e, f^*E|_{C_e})}{H^1(C_e, f^*E|_{C_e})} \right)$ for edges with vertices $v_1$ and $v_2$ at each end:  \begin{align} e(\oplus_e (H^0(C_e, f^* E) - H^1(C_e, f^*E)) &= \prod_e (-1)^{d_e} \prod_{\alpha \in R^+} (t+\Delta_{\alpha})|_{v_1} (t+\Delta_{\alpha})|_{v_2} \\ &= (-1)^{d} \prod_v \prod_{\alpha \in R^+} (t+\Delta_{\alpha})^{val(v)}|_v, \end{align} where $d$ is the total degree of the map $f$. These products are over all positive roots.

The total $\C^* \times T$-equivariant twisting contribution evaluated at a torus-fixed point in the moduli space indexed by graph $\Gamma$ is thus the product over vertices of $\Gamma$  \begin{equation} e(E_{0,m,d}) = (-1)^d \prod_v \prod_{\alpha \in R^+} (t+ \Delta_{\alpha})|_v (t-\Delta_{\alpha})|_v^{1-val(v)}.\end{equation} For $\PP^{n-1}\times \PP^{n-1}$, we can write this as \begin{equation}\label{eq: eqtwisteq} e(E_{0,m,d}) = (-1)^d \prod_{v \in \Gamma} (t+ H_1-H_2)|_v (t- H_1 +H_2)|_v^{1-val(v)}.\end{equation}

\subsection{Weyl classes: exploiting symmetry} To prove that the localization contributions of all graphs with vertices in the diagonal $D$ sum to zero after specialization to the small torus, exploit symmetry: group graphs into classes that have the same contribution $e(N_{\Gamma})$ up to sign after specialization to the small torus $(\C^*)^n$, and then show that the changes of sign coming from the twisting component result in a factor of $t^2$ in the sums of contributions of graphs in these groups. Letting $t=0$ gives the non-equivariant twisting, which is zero.

The Weyl group action on $(\PP^{n-1})^2$ is the action of $S_2$. To group graphs into what we call Weyl-classes, let $S_2$ act on each vertex individually by permuting the vertex labels $\ell_1 \ell_2$. By looking at all graphs obtained from $S_2$ acting on the vertex labels, one can see that some graphs represent a legitimate torus-invariant stable map, while others do not. Thus we define Weyl-classes as follows: 

\begin{defn} A \textit{Weyl-class} $W$ is a set of graphs $\Gamma$ representing torus-invariant $m$-pointed stable maps to $(\mathbb{P}^{n-1})^2$ equivalent under the action of the Weyl group $S_2$ on the vertex labels of $\Gamma$.
\end{defn}

Note that marked points still decorate the vertices of the graphs and are not touched by the $S_2$-action.

The definition of Weyl-class makes sense for all $(\mathbb{P}^{n-1})^k$, but in the case of $(\mathbb{P}^{n-1})^2$ it is particularly simple to illustrate. We ``explode'' a graph $\Gamma$ by removing all vertices $ii$ in the unstable locus $D$, leaving half-edges adjacent to the removed vertices. Then there are $M$ disjoint subgraphs left, which we call $G_1, \ldots, G_M$. All other graphs in the Weyl-class can be obtained by letting $S_2 \times \cdots \times S_2$ ($M$ times) act on the vertex and edge labels of each $G_1, \ldots, G_M$ and reinserting the vertices $ii \in D$, which are invariant under $S_2$. For $(\mathbb{P}^{n-1})^2$, Weyl-classes are small.

\begin{figure}[htpb]
\begin{center}
\includegraphics{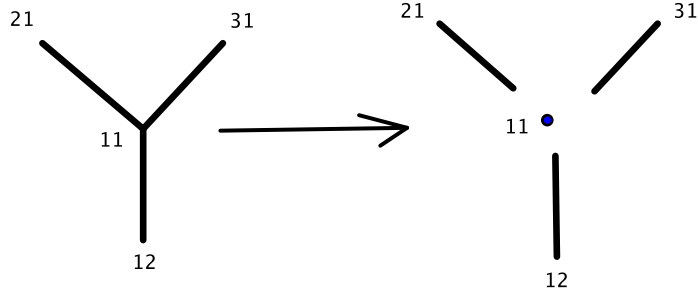}
\end{center}
\caption{An exploded graph}\label{fig: exploded}

\end{figure}In Figure~\ref{fig: exploded}, one can see the process of ``exploding'' a graph. Notice that we can act on each subgraph by $S_2$. Compare this to the Weyl-class we obtain:
\begin{figure}[htp]
\includegraphics[scale=0.68]{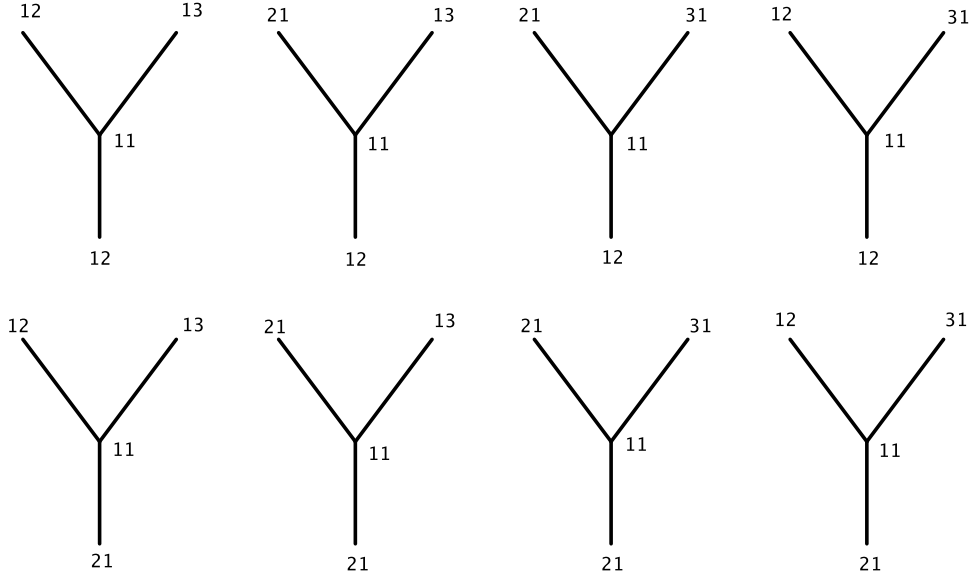}
\caption{Example of a Weyl-class for $(\PP^{n-1})^2$}\label{fig: weylclass}
\end{figure}

\begin{lem} Consider a Weyl-class $W$ of graphs representing torus-fixed loci in $\overline{M}_{0,m}((\PP^{n-1})^2, \{d_i\})$. For all graphs in the same Weyl-class $W$, the contribution of the Euler class of the normal bundle is the same when specialized to the small torus $(\C^*)^n$. That is, for any two graphs $\Gamma, \; \tilde{\Gamma}$ in the same Weyl-class \begin{equation} e(N_{\Gamma})|_{(\mathbb{C}^*)^n} = e(N_{\overline{\Gamma}})|_{(\mathbb{C}^*)^n}. \end{equation}  \end{lem}

\begin{proof} The lemma is clear when the formula for $e(N_{\Gamma})$ is examined. Nontrivial action of $S_2$ on an edge permute homology classes from $d_e [\mathrm{line}]$ in the $i$th factor of $(\PP^{n-1})^2$ to the $j$th factor of $(\PP^{n-1})^2$  and thus changes the superscripts in all edge-dependent terms (but not subscripts). (Total degree $d = \sum_{e} d_e$ remains the same.) Similarly, action of $S_2$ on vertices and flags changes super- but not sub-scripts. Specialization to the small torus is geometrically equivalent to taking the quotient of $(\mathbb{P}^{n-1})^2$ by $S_2$, and formally equivalent to erasing superscripts. Since action of $S_2$ results only in changes of superscript, it is unseen by the action of $(\C^*)^n$. \end{proof}

Similarly, if we let $I_{\Gamma}$ denote the evaluation of insertions $\tilde{\gamma}_1, \ldots, \tilde{\gamma}_m$ at marked points $p_1, \ldots, p_m$ decorating vertices of $\Gamma$, we have 
\begin{lem}\label{lemma: insertions} Consider a Weyl-class $W$ of graphs representing torus-fixed loci in $\overline{M}_{0,m}((\PP^{n-1})^2, \{d_i\})$. For all graphs $\Gamma, \overline{\Gamma} \in W$, \begin{equation} I_{\Gamma}|_{(\mathbb{C}^*)^n} = I_{\overline{\Gamma}}|_{(\mathbb{C}^*)^n}. \end{equation} \end{lem}

\begin{proof} Insertion $\tilde{\gamma}_i$ is evaluated at the vertex $v$ decorated by $p_i$. A vertex is labelled $\ell_1 \ell_2$ and the insertion \[\tilde{\gamma} = \tilde{\sigma}_{\mu} =S_{\mu} (H_1, H_2)\] is evaluated by \[S_{\mu}(H_1, H_2)(v = \ell_1 \ell_2) = S_{\mu} (\lambda^{h_1}_{\ell_1}, \lambda_{\ell_2}^{h_2}).\] Action by $S_2$ on the vertex $v$ will permute the $\ell_i$ (subscripts) but not the $h_i$ (superscripts). $S_{\mu}$ is a symmetric polynomial, so specialization to the small torus (erasing superscripts) is invariant under action of $S_2$. \end{proof}

\section{Vanishing of localization contributions}

Since specialization from $((\C^*)^n)^2$ to $(\C^*)^n$ results in the above-mentioned equalities for $e(N_{\Gamma})$ and $I_{\Gamma}$ for all $\Gamma$ in a Weyl-class, we can factor $\frac{I_{\Gamma}}{e(N_{\Gamma})}|_{(\mathbb{C}^*)^n}$ from the sum of localization contributions over each Weyl-class. What remains is the sum of twisting contributions. In this section, we use the notation $T(-)$ to denote the twisting contribution of any component of a graph $\Gamma$. For $k=2$, this is the product of evaluation of $e(E_{0,m,d}) = \prod_v \prod_{i <j} (t+ H_i-H_j)|_v (t-H_i+H_j)|_v^{1-val(v)}$ at vertices of $\Gamma$ (see equation (\ref{eq: eqtwisteq})). 

\begin{thm}\label{thm: vanishing} The localization contributions $C(\Gamma)$ of all graphs passing through the diagonal $D \subset (\mathbb{P}^{n-1})^2$ cancel, after specialization of the sum $\sum_{\Gamma} C(\Gamma)$ to the small torus $(\C^*)^n$. \end{thm}

\begin{proof} 
Look at one Weyl-class of graphs passing through the diagonal $D$. (Note that if one graph $\Gamma \in W$ passes through the diagonal, all graphs in $W$ pass through the diagonal and at the same vertices.) ``Explode'' the graphs by cutting out vertices in the diagonal, leaving finitely many discrete components which we denote by $G_1, \ldots, G_M$. Each $G_i$ has $m_i$ half-edges where the diagonal was cut away. The contribution of twisting to the localization term can be rewritten \begin{equation} \prod_i T(G_i) \prod_{v \in D} T(v).\end{equation}

As noted above, for all $\Gamma \in W$, $e(N_{\Gamma}), \; I_{\Gamma}$, and $\prod_{v \in D} T(v)$ are the same after specialization. Thus \begin{equation} \sum_{\Gamma \in W} C(\Gamma)|_{(\mathbb{C}^*)^n} = \frac{I_{\Gamma} \prod_{v \in D}T(v)}{e(N_{\Gamma})}|_{(\mathbb{C}^*)^n} \sum_{\Gamma \in W} \prod_{i=1}^M T(G_i)|_{(\mathbb{C}^*)^n}.
\end{equation} 

For $v \in D$, $T(v) = t^{2-val(v)}$ by (\ref{eq: eqtwisteq}). Vanishing of $\sum_{\Gamma \in W} C(\Gamma)$ relies on the fact that $ \left( \prod_{v \in D} T(v) \right)\sum_{\Gamma \in W} \prod_{i=1}^M T(G_i)|_{(\C^*)^n}$ has a positive power of $t$ as an overall factor (to be proved).

The $i$th component of any exploded graph lies entirely in the good locus $U$. For each graph $\Gamma$, label the $i$th component by either $G_i$ or $\overline{G}_i$, where $G_i$ and $\overline{G}_i$ are the two distinct liftings of $\varphi(G_i)$. Then \textit{every} graph in $W$ is specified by describing, for each $i$, whether the $i$th component is $G_i$ or its opposite, $\overline{G}_i$. This is a coin-flipping question -- for each component, heads or tails? Thus we can factor the sum of products into a product of sums: 

\begin{equation}\sum_{\Gamma \in W} \prod_{i=1}^M T(G_i)|_{(\mathbb{C}^*)^n} = \prod_{i=1}^M (T(G_i) +T(\overline{G}_i))|_{(\mathbb{C}^*)^n}.\end{equation}

By definition,  \[T(G_i) + T(\overline{G}_i)  = \prod_{v \in G_i} (t+\Delta_v)(t-\Delta_v)^{1-val(v)} + \prod_{v \in \overline{G}_i} (t+\Delta_v) (t- \Delta_v)^{1-val(v)}.\] Each vertex $jk$ in $G_i$ corresponds to the vertex $kj$ in $\overline{G}_i$, and after specialization, $\Delta_{jk} = -\Delta_{kj}$. Thus\begin{align}T(G_i) + T(\overline{G}_i)  &= \prod_{v \in G_i} (t+\Delta_v)(t-\Delta_v)^{1-val(v)} + \prod_{v \in G_i} (t-\Delta_v) (t+ \Delta_v)^{1-val(v)} \\ 
&=\prod_{v \in G_i} (t^2-\Delta_v^2) \left(\prod \frac{1}{(t-\Delta_v)^{val(v)}} + \prod \frac{1}{(t+\Delta_v)^{val(v)}}\right)
\end{align}

Notice \begin{equation}  \prod_v \frac{1}{(t-\Delta_v)^{val(v)}}+ \prod_v \frac{1}{(t+\Delta_v)^{val(v)}} = \frac{\prod_v (t-\Delta_v)^{val(v)}+\prod_v (t+ \Delta_v)^{val(v)}}{\prod_v (t-\Delta_v)^{val(v)}(t+\Delta_v)^{val(v)}}. \end{equation} When $\sum_v val(v)$ is odd, the numerator of this rational expression has no constant term. Thus, $T(G_i) + T(\overline{G}_i)$ is divisible by $t$. What remains is to prove that there are more $G_i$ with $\sum_{v \in G_i} val(v)$ odd than the order of the pole, $\sum_{v \in D}(val(v)-2)$.

In fact, we can do something easier: we can prove that the number of $G_i$ with just \textit{one} half-edge is greater than $\sum_{v \in D} (val(v) -2)$. (One half-edge guarantees that $\sum_{v \in G_i} val(v)$ is odd.) Let $\nu_j$ denote the number of components $G_i$ with $j$ half-edges. For any graph $\Gamma$, $\sum_j j \nu_j = \sum_{v \in D} val(v)$. By induction, $|\{ v \in D \} | = 1+ \sum_{j=2}(j-1) \nu_j$.

The degree of $t$, then, is greater or equal to 
\begin{align} \nu_1 + \sum_{v \in D}(2-val(v)) &= \nu_1 + 2(1+\sum_{j=2}^{\infty} (j-1) \nu_j - \sum_{v \in D}val(v) \\ &=  \nu_1 + 2(1+\sum_{j=2}^{\infty} (j-1) \nu_j - \sum_{j=1}^{\infty} j \nu_j \\ &= \nu_1 +2 + \big(\sum_{j=2}^{\infty} (2j-2-j) \nu_j \big) - \nu_1 \\ &= 2+ \sum_{j=2}^{\infty} (j-2)\nu_j 
\end{align}
Since $\nu_j \geq 0$ because it is enumerative, and $j-2$ appears only for $j \geq 2$, this quantity is always greater than or equal to 2. Thus we always have a factor of $t^2$ in the sum of localization contributions over a Weyl class whose graphs include a vertex label in the diagonal. Letting $t \rightarrow 0$ shows that the sum of localization contributions over such a Weyl class equals zero. \end{proof}

Combined with Theorem (\ref{thm: vanishing}), this implies the following:

\begin{thm} Let $\Gamma$ be graphs indexing $((\C^*)^n)^2$-fixed loci in $\overline{M}_{0,m}((\PP^{n-1})^2,(d_1,d_2))$ and let $\mathfrak{G}$ be graphs indexing $(\C^*)^n$-fixed loci in $\overline{M}_{0,m}(Gr(2,n),d)$. Consider all pairs $(d_1,d_2)$ such that $d_i \geq 0$ and $d_1+d_2=d$. Then \begin{displaymath} \sum_{\mathfrak{G}} \frac{I_\mathfrak{G}}{e(N_{\mathfrak{G}})} = \frac{1}{2} \left( \sum_{\Gamma} \frac{I_{\Gamma}}{e(N_{\Gamma})} T(\Gamma) \right)|_{(\C^*)^n}.  \end{displaymath}\end{thm}  Combined with the implications of Lemma \ref{lemma: equality} \begin{displaymath}
\langle \gamma_1, \ldots, \gamma_m \rangle_{0,m,d}^{Gr(2,n)} = \frac{1}{2} \sum_{\tilde{d} \mapsto d} \langle \tilde{\gamma_1}, \ldots, \tilde{\gamma_m} \rangle_{0,m,\{d_i\}}^{(\mathbb{P}^{n-1})^2, E}
\end{displaymath} which is our earlier Theorem \ref{thm: main}.

\section{Further research and extensions}
\subsection{Case where $k$ is greater than 2}
Unfortunately, this method becomes combinatorially very complicated for $Gr(k,n)$ with $k>2$. In particular, if one sums over Weyl-classes the degree of the pole in the contribution grows factorially with number of edges while the degree of the zero grows linearly. It is possible that a clever use of combinatorics might step around this problem. Currently, though, no method of attack has proved fruitful.

\subsection{Higher genus}
Higher genus is another situation of interest. Looking at genus zero and $d_1 + \ldots + d_k = d$, we had \begin{align} \begin{split}\dim \overline{M}_{0,m}((\mathbb{P}^{n-1})^k, (d_1, \ldots, d_k)) \\ - \dim \overline{M}_{0,m}(Gr(k,n), d)\end{split} \nonumber\\ &= (n \left(\sum d_i \right) + k(n-1)- 3 + m)-(nd + k(n-k)- 3 + m ) \nonumber\\ &=  k(n-1) -k(n-k) \nonumber\\ &= k^2-k. \end{align} For genus one, we have instead  \begin{align} \dim \overline{M}_{1,m}((\mathbb{P}^{n-1})^k, (d_1, \ldots, d_k)) \nonumber\\ - \dim \overline{M}_{1,m}(Gr(k,n), d)  \nonumber\\ &= (n \left(\sum d_i \right)+ m)-(nd + m ) \nonumber\\ &= 0. \end{align} There is no difference in dimensions of the moduli space, so the twisting bundle has expected dimension zero. For genus greater than one, the expected rank of the twisting bundle is negative. The negativity of rank means that the conjecture does not generalize to higher genus in a straightforward manner.

\paragraph{Acknowledgements} Thanks to Ionu\c{t} Ciocan-Fontanine, my thesis advisor, who suggested this problem and several more. I also benefited from conversation with Ezra Miller, Alexander Voronov, Igor Pak, and Victor Reiner.

\bibliographystyle{amsalpha}

\bibliography{bigbibliography}

\end{document}